\documentclass[12pt]{amsart}
\usepackage[lite, alphabetic, nobysame]{amsrefs}
\usepackage{amsmath,mathtools}
\usepackage{amssymb}
\usepackage{srcltx}
\usepackage{xcolor}

\newcommand{\MB}{\boldsymbol{M}}       
\newcommand{\NB}{\boldsymbol{N}}       
\newcommand{\AB}{\boldsymbol{A}}       
\newcommand{\BB}{\boldsymbol{B}}       
\newcommand{\CB}{\boldsymbol{C}}
\newcommand{\DB}{\boldsymbol{D}}

\newcommand{\KB}{\boldsymbol{K}}

\newcommand{\N}{\mathbb{N}}

\newcommand{\KC}{\mathcal{K}}

\newcommand{\LC}{\mathcal{L}}

\newcommand{\Aut}{\mathrm{Aut}}
\newcommand{\Homeo}{\mathrm{Homeo}}

\newcommand{\acts}{\curvearrowright}

\newcommand{\Fraisse}{Fra\"iss\'e}

\newtheorem{theorem}{Theorem}[section]
\newtheorem{lemma}{Lemma}

\newtheorem{proposition}[theorem]{Proposition}

\newtheorem*{lemma*}{Lemma}

\theoremstyle{definition}

\newtheorem*{definition*}{Definition}
\newtheorem*{Ack}{Acknowledgements}

\theoremstyle{remark}

\newenvironment{theorem*}[1][]{\par\medskip\noindent\textbf{Proposition} #1\textbf{.}\itshape}{\medskip}

\newenvironment{theorem**}[1][]{\par\medskip\noindent\textbf{Theorem} #1\textbf{.}\itshape}{\medskip}

\author{Aristotelis Panagiotopoulos}
\address{Department of Mathematics, 1409 W. Green st., University of Illinois, Urbana, IL 61801, USA}
\email{panagio2@illinois.edu}

\begin{document}
\title{Compact spaces as quotients of projective Fra\"{I}ss\'E limits}

\maketitle
\begin{abstract}
We develop a theory of projective \Fraisse{} limits in the spirit of Irwin-Solecki. The structures here will additionally support dual semantics as in \cite{SoleckiRamsey1,SoleckiRamsey2}. Let $Y$ be a compact metrizable space and let $G$ be a closed subgroup of $\Homeo(Y)$. We show that there is always a projective \Fraisse{} limit $\KB$  and a closed equivalence relation $\mathfrak{r}$ on its domain $K$ that is definable in $\KB$,  so that the quotient of $K$ under $\mathfrak{r}$ is homeomorphic to $Y$ and the projection $K\to Y$ induces a continuous group embedding 
$\Aut(\KB)\hookrightarrow G$ with dense image.

\end{abstract}

\section*{Introduction}
The interplay between \Fraisse{} theory and dynamics of non-Archimedean Polish groups established by the correspondence between closed subgroups of $S_\infty$ and countable \Fraisse{} structures has given rise to a reach theory with various interesting results; see for example the survey \cite{survey}. Taking this fact as a departure point it is natural to seek  what would the natural correspondence be for projective \Fraisse{} structures.

Projective \Fraisse{} structures were introduced by T. Irwin and S. Solecki in \cite{Irwin} and since then they  have been used by several authors in the study of the dynamics of compact spaces such as the pseudo-arc, the Lelek fan and the Cantor space;  see \cite{Irwin}, \cite{Lelek}, \cite{Ample}, \cite{Large}. There is a standard process that is implemented in all these papers. One starts with a compact metrizable space $Y$ where $Y$ or $\Homeo(Y)$ is under investigation. Then one defines an appropriate class $\KC$ of finite model theoretic $\LC$-structures  which ``approximate'' $Y$. If this class $\KC$ satisfies the {\em projective \Fraisse{} axioms}  then the projective \Fraisse{} limit $\KB$ of $\KC$ is uniquely defined and it is a compact, zero-dimensional metrizable space $K$ together with closed interpretations for the model-theoretic content $\mathcal{L}$; see \cite{Irwin}, or Section \ref{Section: projective Fraisse structures}. In all cases above, whenever $Y$ is not totally disconnected \cite{Irwin,Lelek}, the language $\mathcal{L}$ contained---implicitly or explicitly---a special binary predicate $\mathfrak{r}$ whose interpretation in all finite structures of $\mathcal{K}$ is a reflexive, symmetric graph, and whose interpretation in $\KB$ is a closed equivalence relation. Moreover, the space $K/\mathfrak{r}$ is homeomorphic to $Y$ and the quotient projection map $K\mapsto Y$
induces a continuous group embedding $\Aut(\KB)\hookrightarrow \mathrm{Homeo}(Y)$ whose image is dense in $\mathrm{Homeo}(Y)$.
This correspondence between $\KB$ and $Y$ allows one to study $Y$ and $\mathrm{Homeo}(Y)$ through the combinatorial properties of the finite structures in $\mathcal{K}$.

The purpose of this note is to turn these heuristics into a general theorem. The price we have to pay is that we have to include dual predicates in our language $\mathcal{L}$ and use them to further endow our structures with dual semantics. Dual predicates and structures were first introduced by S. Solecki in \cite{SoleckiRamsey1,SoleckiRamsey2}, where they were used to institute a uniform structural approach to various Ramsey theoretic results. Dual predicates in $\LC$ quantify over dual tuples (finite clopen partitions). We will call a $\LC$-structure purely dual if $\LC$ contains only dual predicates.

In Section \ref{Section: projective Fraisse structures}, we rewrite the projective \Fraisse{} theory developed in \cite{Irwin} so that it includes the dual semantics found in \cite{SoleckiRamsey1,SoleckiRamsey2}.  In Section \ref{Section: Closed subgroups of homeomorphism groups}, we use a standard orbit completion argument and equipping $\Homeo(K)$ with the compact-open topology we get the following proposition.

\begin{theorem*}[\textbf{\ref{proposition orbits}}]
Let $G$ be a closed subgroup of $\Homeo(K)$ where $K$ is zero-dimensional, compact, metrizable space. Then there is a purely dual projective \Fraisse{} structure $\KB$ on domain $K$ such that $\Aut(\KB)=G$. 
\end{theorem*}   

Proposition \ref{proposition orbits} is essentially the dual of the statement that every closed subgroup of $S_\infty$ is the automorphism group of a \Fraisse{} limit on domain $\mathbb{N}$. In Section  \ref{Section: Turning a structure direct to dual} we show how to turn any topological $\LC$-structure into a purely dual one without losing any information. We also show that the above theorem is false if we do not allow our structures to support a dual structure. Therefore the context of dual structures is strictly more general than the context of direct structures. 

 In Section \ref{Section: Compact spaces},  we fix a special binary  relation symbol $\mathfrak{r}$ whose interpretation will always be a reflexive and symmetric closed relation. This should be paralleled with the  metric \Fraisse{} theory where the symbol $d$ is  reserved as a signifier for the metric.  A formal relational language $\LC$ will be decorated with the subscript $\mathfrak{r}$ whenever $\mathfrak{r}$ belongs to $\LC$. Therefore, we always have  $\mathfrak{r}\in\LC_\mathfrak{r}$ and $\LC_\mathfrak{r}$-structures are, in particular,  reflexive $\mathfrak{r}$-graphs. We say that an $\LC_\mathfrak{r}$-structure $\KB$ is a pre-space if $\mathfrak{r}^{\KB}$ is moreover transitive and therefore an equivalence relation. We apply Proposition \ref{proposition orbits}  to get the following result.

\begin{theorem**}[\textbf{\ref{compactSpacesMainTheorem}}]
Let $G$ be a closed subgroup of $\Homeo{}(Y)$, for some compact metrizable space $Y$. Then there is a
projective \Fraisse{} pre-space $\KB$ such that $K/\mathfrak{r}^{\KB}$ is homeomorphic to $Y$, and the
quotient projection
\[ K \mapsto Y\]
induces a continuous group embedding $\Aut(\KB)\hookrightarrow  G$,  with dense image in $G$. 
\end{theorem**}

We shall note here that R. Camerlo characterized all possible quotients $M/\mathfrak{r}^{\MB}$ of \Fraisse{} structures in the language $\{\mathfrak{r}\}$ to be certain combinations of singletons, Cantor spaces and  pseudo-arcs \cite{Camerlo}.

 \begin{Ack}
 This set of notes includes --I think-- the first written account of \Fraisse{} theory with dual predicates. This is not however the original idea of this paper. The idea of using dual predicates was first introduced in a systematic way by S. Solecki \cite{SoleckiRamsey1, SoleckiRamsey2} in Ramsey theory to formalize dual Ramsey statements such as the Graham and Rothschild theorem \cite{Graham}. S. Solecki had known for a fact that projective  \Fraisse{} theory can be developed including dual predicates and he was the one who pointed out that dual structure might be needed for the proof of Proposition \ref{proposition orbits}. An example in Section \ref{Section: Turning a structure direct to dual} shows that for this theorem to be true dual predicates  are actually necessary. I would like to thank S\l{}awek since many of the ideas included here grew out of the discussions we had. I would also like to thank Ola Kwiatkowska for sharing with me her insight on this problem, as well as Alex Kruckman and Gianluca Basso for their useful comments.
 \end{Ack}

\section{Preliminaries}\label{Section Preliminaries}

In what follows, $K$ will always denote a zero-dimensional, compact, metrizable space. Our main objects of study will be spaces $K$ as above, which support both usual model-theoretic structure as well as dual structure. To make this precise, following S. Solecki \cite{SoleckiRamsey1,SoleckiRamsey2}, we consider the following two types of tuples. 
\begin{itemize}
\item In the classical model-theoretic context, a tuple of size $n>0$ in $K$ corresponds to an injection
\[i:\{0,\ldots,n-1\}\to K \]    

We will call this kind of tuple a \emph{direct tuple}. We denote the set of all direct tuples in $K$ of size $n>0$ by $K^{[n]}$.  

\item In the dual context, a tuple of size $n$ in $K$ corresponds to a surjection
\[e: K \to \{0,\ldots,n-1\} \]    
Since our intention is to work with topological structures, we endow $\{0,\ldots,n-1\}$ with the discrete topology and we impose a further regularity condition, that $e$ is also continuous. We will call this kind of tuple a \emph{dual tuple}. We denote the set of all dual tuples in $K$ of size $n$ by $[n]^K$.  
\end{itemize}

Notice that for $K$ zero-dimensional, compact, metrizable space, the set $[n]^K$ is at most countably infinite and moreover these functions suffice to separate points of $K$, i.e., for every $x_1,x_2\in K$ there is an $e\in[n]^K$ such that $e(x_1)\neq e(x_2)$. Notice also that the set $[n]^K$ of all dual tuples naturally corresponds to the set $\mathrm{CP}_n(K)$ of all clopen, ordered, $n$-partitions of $K$, i.e.,
\[\mathrm{CP}_n(K)=\big\{(\Delta_0,\ldots,\Delta_{n-1}) : \Delta_i\subset K\,\,\, \text{clopen},\,\,\,\Delta_i\cap\Delta_j=\emptyset\,\,\,\cup_i\Delta_i= K \big\}.\]

Whenever it is convenient for notational purposes we will not distinguish between the set $\{\Delta_0,\ldots,\Delta_{n-1}\}$ and the tuple $(\Delta_0,\ldots,\Delta_{n-1})$.  If for example $P\in\mathrm{CP}_n(K)$ and $\Delta$ is a clopen set appearing some of the $n$ entries of $P$ then we will write $\Delta\in P$.

\subsection{Topological $\LC$-structures}

We will work with relational only languages $\LC$. The structures we describe here, with additional dual function symbols, were introduced in \cite{SoleckiRamsey1,SoleckiRamsey2}. To each relational symbol $R$ in $\LC$ corresponds some natural number $\mathrm{arity}(R)>0$ which is the arity of the symbol $R$. Moreover, for every symbol in $\LC$ we have predetermined our intention to use it in the direct or in the dual context. We will make the convention here of using lower case letters $r,p,q,\ldots$ for direct relational symbols, and capital letters $R,P,Q,\ldots$ for dual relational symbols.  We will call the language $\LC$ \emph{purely direct} if it contains only direct symbols and \emph{purely dual} if it contains only dual symbols.

By a \emph{topological $\LC$-structure} $\KB$ we mean a zero-dimensional, metrizable, compact space $K$ together with appropriate interpretations for every symbol in $\LC$.
\begin{itemize}
\item If $r\in\LC$ is a direct relation symbol of arity $n$, an appropriate interpretation for $r$ is any closed subset $r^{\KB}$ of $K^{[n]}$.
\item If $R\in\LC$ is a dual relation symbol of arity $n$, an appropriate interpretation for $R$ is any subset $R^{\KB}$ of $[n]^K$, or equivalently, any subset of $\mathrm{CP}_n(K)$. 
\end{itemize}

We call a topological $\LC$-structure \emph{purely direct $\LC$-structure} whenever $\LC$ is purely direct and \emph{purely dual $\LC$-structure} whenever $\LC$ is purely dual.

\subsection{Morphisms}

Following \cite{Irwin}, we will be working with epimorphisms. The epimorphisms here will additionally preserve the dual structure. Such epimorphisms were introduced in \cite{SoleckiRamsey1,SoleckiRamsey2}.

 Let $\AB,\BB$ be two dual topological $\LC$-structure. By an \emph{epimorphism} $f$ from $\AB$ to $\BB$ we mean a continuous surjection $f:A\to B$  such that:
\begin{itemize}
\item for every $r\in\LC$ of arity say $m$ and every $\beta\in B^{[m]}$ we have
\[\beta\in r^{\BB}\quad\Longleftrightarrow\quad \exists \alpha\in r^{\AB}\quad \beta=f\circ\alpha \]

\item for every $R\in\LC$ of arity say $m$ and for every $\beta\in [m]^B$ we have
\[\beta\in R^{\BB}\quad\Longleftrightarrow\quad \beta \circ f\in R^{\AB} \]

\end{itemize}

An isomorphism between $\AB$ and $\BB$ is a bijective epimorphism and an automorphism of $\AB$ is an isomorphism from $\AB$ to $\AB$.

Let $\KB$ be a topological $\LC$-structure, let $A$ be a zero-dimensional, metrizable, compact space and let $f:K\to A$ be a continuous surjection. Notice that there is a unique topological $\LC$-structure $\AB$ on domain $A$ that renders $f$ an epimorphism. We call this structure $\AB$, the structure \emph{induced by the map} $f$.

Let $f:\KB\to \AB, h:\KB\to\BB$ be epimorphisms. We say that $f$ \emph{factors through} $h$ if there is an epimorphism $f_h:\BB\to\AB$ such that $f_h\circ h= f$.

\begin{lemma}\label{factors through lemma}
Let $\AB,\BB$ and $\KB$ be topological $\LC$-structures with $\AB,\BB$ finite. Let also $f:\KB\to\AB$ and  $g:\KB\to\BB$ be epimorphisms. Then there is a finite topological $\LC$-structure $\CB$ and an epimorphism $h:\KB\to\CB$ such that both $f$ and $g$ factor through $h$. 
\end{lemma}
\begin{proof}
Let $C=(\Delta_0,\ldots,\Delta_{n-1})$ a clopen partition of $K$ whose every entry $\Delta_i$ is a (nonempty) set of the form $f^{-1}(a)\cap g^{-1}(b)$, where $a\in A$ and $b\in B$. Let $h:K\to C$ be the inclusion map, i.e, $h(x)=\Delta_i$ if and only if $x\in\Delta_i$. This map is a continuous surjection, so, it induces a structure $\CB$ on domain $C$. It is immediate now that both $f$ and $g$ factor through $h$.   
\end{proof}

Given a sequence $\AB_1,\AB_2,\ldots,\AB_i,\ldots$ of finite topological $\LC$-structures together with epimorhisms $\pi_i:\AB_{i+1}\to\AB_i$, we can define a new structure $\MB$ and epimorphisms $\pi^\infty_i:\MB\to\AB_i$ through an inverse limit construction. Let
\[M=\big\{(a_1,a_2,\ldots)\in\prod_{i\in\N}A_i : \,\,\,\forall\,i\geq 1\,\,\, \pi_i(a_{i+1})=a_i\big\}.\]
$M$ is a closed subset of the compact space $\prod_{i\in\N}A_i$ and it will serve as the domain of $\MB$. We define $\pi^\infty_i$ to be the projection map from $M$ to $A_i$. 

For $r\in\LC$ of arity say $m$, and $\beta\in M^{[m]}$ we let $\beta\in r^{\MB} $ if and only if $\pi^\infty_i\circ\beta\in r^{\AB_i}$ for all $i\in\N$. For $R\in\LC$ of arity say $m$, and $\gamma\in [m]^M$, notice that there is an $i_0\in\N$ such that $\gamma$ factors through $\pi^\infty_{i_0}$. Let  $\alpha\in [m]^A$ be such that $\gamma=\alpha\circ\pi^\infty_{i_0}$. We let $\gamma\in R^{\MB}$ if and only if $\alpha\in R^{\AB_{i_0}}$, which happens if and only if for every $i>i_0$ we have $(\alpha\circ\pi_{i_0}\circ\ldots\circ\pi_{i-1})\in R^{\AB_{i}}.$ 

This turns $\MB$ into a topological $\LC$-structure and every $\pi^\infty_i$ to an epimorphism. We call $\MB$ the \emph{inverse limit} of the \emph{inverse system} $\{(\AB_i,\pi_i): i\in\N\}$ and we write $\MB=\varprojlim(\AB_i,\pi_i).$

\section{Projective \Fraisse{} structures} \label{Section: projective Fraisse structures}
In Chapter $7$ of \cite{Hodges}, Hodges reviews the theory of \Fraisse{} limits of direct structures via direct morphisms (embeddings). Following Hodges and \cite{Irwin} we present here the theory of \Fraisse{} limits of topological $\LC$-structures via dual morphisms. To avoid confusion, we should emphasize two things. First, what in \cite{Irwin} is called topological $\LC$-structure, here it falls under the name purely direct topological $\LC$-structure. Secondly, in contrast with the definition that we will be using here, in \cite{Irwin} a projective \Fraisse{} class  is not bound  to satisfy the hereditary property ($\mathrm{HP}$).   

We say that a topological $\LC$-structure $\MB$ is \emph{projectively \Fraisse{}} or \emph{projectively ultra-homogeneous} if for every two epimorphisms $f_1,f_2$ of $\MB$ on some finite topological $\LC$-structure $\AB$ there is an automorphism $g$ of $\MB$ such that $f_1\circ g=f_2$.

For every topological $\LC$-structure $\MB$ we denote by $\mathrm{Age}(\MB)$ the class of all the finite topological $\LC$-structures $\AB$ such that $\MB$ epimorphs on $\AB$. We call a class $\KC$ of topological $\LC$-structures an \emph{age} if $\KC=\mathrm{Age}(\MB)$ for some topological $\LC$-structure $\MB$. It is immediate that if $\KC$ is an age, then $\KC$ is not empty, any subclass of $\KC$ of pairwise non-isomorphic structures is at most countable, and the following two properties hold for $\KC$.
\begin{itemize}
\item \emph{Hereditary Property} ($\mathrm{HP}$): if $\AB\in\KC$ and $\AB$ epimorphs onto a structure $\BB$, then $\BB\in\KC$.
\item \emph{Joint Surjecting Property} ($\mathrm{JSP}$): if $\AB,\BB\in\KC$ then there is $\CB\in\KC$ that epimorphs onto both $\AB$ and  $\BB$.
\end{itemize}

The converse is also true i.e. if $\KC$ is a non empty class of finite topological $\LC$-structures such that any subclass of $\KC$ of pairwise non-isomorphic structures is at most countable and the above two properties hold for $\KC$ then  $\KC$ is an age. 

To see this, let $\AB_1,\AB_2,\ldots$ be a list of structures in $\KC$ that up to isomorphism  exhaust $\KC$. Using the $\mathrm{JSP}$ we can find a new list $\BB_1,\BB_2,\ldots$ of structures in $\KC$ such that $\BB_1=\AB_1$ and for $i>1$, $\BB_{i+1}$ epimorphs on both $\AB_{i+1}$ and $\BB_{i}$. Let $\pi_i:\BB_{i+1}\to\BB_i$ be such epimorphisms and let $\MB=\varprojlim(\BB_i,\pi_i)$. Then $\KC=\mathrm{Age}(\MB)$ because by construction $\MB$ epimorphs to every $\AB_i$ and moreover, every epimorphism of $\MB$ to some finite dual topological $\LC$-structure $\AB$ factors through an epimorphism from some $\BB_i\in\KC$ that was used in the inverse system so, by $\mathrm{HP}$ we have that $\AB\in\KC$.

Given now that the structure $\MB$ is projectively ultrahomogeneous, it is easy to see that its age $\KC$ satisfies moreover the following property.
\begin{itemize}
\item \emph{Projective Amalgamation Property} ($\mathrm{PAP}$): if $\AB,\BB,\CB\in\KC$ and $f_A:\AB\to\CB$, $f_B:\BB\to\CB$ are epimorphisms, then there is $\DB\in\KC$ and epimorphisms $g_A:\DB\to\AB$, $g_B:\DB\to\BB$ such that $f_A\circ g_A= f_B\circ g_B$.
\end{itemize}

To check that this is true, notice first that since $\AB,\BB\in\KC$, there are epimorphisms $h_A:\MB\to\AB$ and $h_B:\MB\to\BB$. But then $f_A \circ h_A$ and $f_B \circ h_B$ are both epimorphisms from $\MB$ to $\CB$. So, by projective ultra-homogeneity of $\MB$ there is $\phi\in\Aut{}(\MB)$ such that $f_A \circ h_A\circ\phi=f_B \circ h_B$. Using Lemma \ref{factors through lemma}, we can find $\DB\in\KC$ and an epimorphism $h_D:\MB\to\DB$ such that $h_A$ and $h_B\circ\phi^{-1}$ factor through $h_D$. Let $g_A:\DB\to\AB$ and $g_B:\DB\to\BB$ be the maps that close these diagrams, i.e., $g_A\circ h_D=h_A$ and $g_B\circ h_D=h_B\circ\phi^{-1}$.  The functions $g_A$ and  $g_B$ are the required epimorphisms from $\DB$ to $\AB$ and $\BB$ in respect. 

In Theorem \ref{theorem dual Fraisse} we will see that the converse is also true, i.e., if an age $\KC$ has $\mathrm{PAP}$ then we can built from it a projective \Fraisse{} structure $\MB$ with $\mathrm{Age}(\MB)=\KC$. An age $\KC$ that satisfies $\mathrm{PAP}$ is called \emph{projective \Fraisse{} class}.

Let $\MB$ be a topological $\LC$-structure with $\mathrm{Age}(\MB)=\KC$. We say that $\MB$ has the \emph{finite extension property} if for every $\AB,\BB\in\KC$ and $f:\BB\to\AB$, $g:\MB\to\AB$ epimorphisms, there is an epimorphism $h:\MB\to\BB$ such that $f \circ h=g$. We say that $\MB$ has the \emph{one point extension property} if the above holds when the size of $B$ is one more than the size of $A$. Notice that for any topological $\LC$-structure $\MB$, $\MB$ has the one point extension property if and only if $\MB$ has the finite extension property.

\begin{lemma}\label{finite extension to isomorphism lemma}
Let $\MB$ and $\NB$ be two topological $\LC$-structure of the same age $\KC$. Let $\AB\in\KC$ and let $f:\MB\to\AB$ and $g:\NB\to\AB$ be two epimorphisms. If $\MB$ and $\NB$ have the finite extension property then there is an isomorphism $h:\MB\to\NB$ such that $g\circ h= f$. 
\end{lemma}
\begin{proof}
 
We will use a back and forth type of argument. For every $n\in\N$ we will construct $\AB_n\in\KC$ and epimorphisms $f_n:\MB\to \AB_n$, $g_n:\NB\to \AB_n$, and for every $n>0$ we will also construct an epimorphism $\pi_{n-1}:\AB_n\to \AB_{n-1}$.  At the end of the construction $\MB$ and $\NB$ will be proven to be isomorphic to $\varprojlim(\AB_n,\pi_n)$. By using these indirect isomorphisms we will get the desired isomorphism $h$.  Let $\{e_n: n\in\N\}$ be an enumeration of dual tuples $[m]^M$ of $M$ for every $m>0$ and let $\{e'_n: n\in\N\}$ be an enumeration of dual tuples $[m]^N$ of $N$ for every $m>0$.

$\mathit{n=0}.$ Let $\AB_0=\AB$, $f_0=f$ and  $g_0=g$.

\emph{odd} $\mathit{n>0}$. Using Lemma \ref{factors through lemma} we can find a structure $\AB_n$ and an epimorphism $f_n:\MB\to\AB_n$ such that both $f_{n-1}$ and $e_{n-1}$ factor through $f_n$. Let $\pi_{n-1}:\AB_n\to\AB_{n-1}$ be the epimorphism that closes the one diagram, i.e., $\pi_{n-1}\circ f_{n}=f_{n-1}$. Finally define $g_n:\NB\to\AB_n$ to be any map such that $\pi_{n-1}\circ g_n=g_{n-1}$. A map like this exists, since $\NB$ satisfies the finite extension property.   

\emph{even} $\mathit{n>0}$. Again, using Lemma \ref{factors through lemma} we can find a structure $\AB_n$ and an epimorphism $g_n:\NB\to\AB_n$ such that both $g_{n-1}$ and $e'_{n-1}$ factor through $g_n$. Let $\pi_{n-1}:\AB_n\to\AB_{n-1}$ be the epimorphism that closes the one diagram, i.e., $\pi_{n-1}\circ g_{n}=g_{n-1}$. Finally define $f_n:\MB\to\AB_n$ to be any map such that $\pi_{n-1}\circ f_n=f_{n-1}$. A map like this exists, since $\MB$ satisfies the finite extension property.   

Let now $\BB=\varprojlim(\AB_n,\pi_n)$. The maps $\mu:\MB\to\BB$ with $\mu(x)=(f_0(x),f_1(x),\ldots)$ and $\nu:\NB\to\BB$ with $\nu(x)=(g_0(x),g_1(x),\ldots)$ are bijections since the families $\{e_n\}$ and $\{e'_n\}$ separate points of $M$ and $N$ in respect. It is moreover easy to see that $\mu$ and $\nu$ are actually isomorphisms. So, the map $h:\MB\to\NB$ with $h=\nu^{-1}\circ\mu$ is also an isomorphism which by construction satisfies the desired property  $g\circ h= f$.
\end{proof}

\begin{lemma}\label{1 point=fraisse}
Let $\MB$ be a topological $\LC$-structure with $\mathrm{Age}(\MB)=\KC$ then the following are equivalent:
\begin{enumerate}
\item $\MB$ is projectively ultrahomogeneous;
\item $\MB$ has the finite extension property;
\item $\MB$ has the one point extension property.
\end{enumerate}
\end{lemma}
\begin{proof}
It is immediate that (2) and (3) are equivalent. We prove that (1) is also equivalent to (2). 
\item[$(1)\to(2)$] Let $\AB,\BB\in\KC$ and $f:\BB\to\AB$, $g:\MB\to\AB$ epimorphisms. Since $\BB\in\KC=\mathrm{Age}(\MB)$, there is an epimorphism $j:\MB\to\BB$. So, $f\circ j:\MB\to\AB$ is an epimorphism, and by the projective ultra-homogeneity of $\MB$ there is $\phi\in\Aut{}(\MB)$ with $g\circ\phi=f\circ j$. Let $h=j\circ\phi^{-1}$. Then $h:\MB\to\BB$ is an epimorphism such that $f \circ h=g$. 
\item[$(2)\to(1)$]
Let $f_1,f_2:\MB\to\AB$ be epimorphisms for some $\AB\in\KC$. Then by Lemma \ref{finite extension to isomorphism lemma}, there is $g\in\Aut{}(\MB)$ such that $f_1\circ g=\ f_2$.
\end{proof}

\begin{theorem}\label{theorem dual Fraisse}
For every projective \Fraisse{} class $\KC$ there is a unique, up to isomorphism, projectively ultra-homogeneous topological $\LC$-structure $\MB$ such that $\mathrm{Age}(\MB)=\KC$.
\end{theorem}
\begin{proof}
First notice that is $\MB_1,\MB_2$ share the same age and are both projectively ultra-homogeneous, by Lemma \ref{1 point=fraisse} they have finite extension property. Let $\AB$ any structure in $\KC$. Since $\KC$ is the age of both $\MB_1,\MB_2$, there are epimorphisms $f_1:\MB_1\to\AB$ and $f_2:\MB_2\to\AB$. Lemma \ref{finite extension to isomorphism lemma} gives as then an isomorphism $h$ between $\MB_1$ and $\MB_2$. 
\end{proof}

\section{Closed subgroups of $\Homeo{}(K)$}\label{Section: Closed subgroups of homeomorphism groups}

By definition and since $K$ is compact, every automorphism of a topological $\LC$-structure $\KB$ is also a homeomorphism, therefore, $\Aut{}(\KB)$ can be seen as a subgroup of $\Homeo{}(K)$. We will view $\Homeo(K)$ as a topological group equipped with the compact-open topology $\tau_{\mathrm{co}}$. The collection of the sets 
\[V(F,U)=\{g\in\Homeo(K): g(F)\subset U\},\]
where $F$ is a compact subset of $K$ and $U$ is an open subset of $K$, provide a subbase for $\tau_{\mathrm{co}}$. In this topology the group $\Aut{}(\KB)$ of automorphisms of a dual topological $\LC$-structure $\KB$ is a closed subgroup of $\Homeo(K)$. To check this, let $g\not\in\Aut{}(\KB)$. We will find an open neighborhood $V_g$ of $g$ in $\Homeo(K)$ which does not intersect $\Aut{}(\KB)$. Since $g\not\in\Aut{}(\KB)$, one of the following holds:
\begin{enumerate}
\item there is $R\in\LC$ of arity say $m$ and a dual tuple $e\in [m]^K$ such that
$\KB\models R(e)$ if and only if $\KB\not\models R(e\circ g^{-1})$, or
\item there is $r\in\LC$ of arity say $m$ and a tuple $i\in K^{[m]}$ such that
$\KB\models r(i)$ but $\MB\not\models r(g\circ i)$, or
\item there is $r\in\LC$ of arity say $m$ and a tuple $i\in K^{[m]}$ such that
$\KB\not\models r(i)$ but $\MB\models r(g\circ i)$.
\end{enumerate}

In the first case notice that if $g\not\in\Aut{}(\KB)$ then there is $R\in\LC$ of arity say $m$ and  
But then $V\big(e^{-1}(0),g\circ e^{-1}(0) \big)\cap\ldots\cap V\big(e^{-1}(m-1),g\circ e^{-1}(m-1)\big)$ is an open subset of $\Homeo(K)$ containing $g$ and lying entirely out of $\Aut{}(\KB)$.

In the second case, because $r^{\MB}$ is closed, we can find an open rectangle $U_0\times\ldots\times U_{m-1}$ around $\big((g\circ i)(0),\ldots,(g\circ i)(m-1)\big)$ which does not intersect $r^{\MB}$. Therefore, let $V_g=V\big(\{i(0)\},U_0\big)\cap\ldots\cap V\big(\{i(m-1)\},U_{m-1}\big)$.

For the last case, notice that if we let $(b_0,\ldots,b_{m-1})=\big((g\circ i)(0),\ldots, (g\circ i)(m-1)\big)$, then, as in the previous case we can find open neighborhood $V_{g^{-1}}$ of $g^{-1}$ such that for every $f\in V_{g^{-1}}$, $f(b_0,\ldots,b_{m-1})\not\in r^{\MB}$. Let then $V_g=V_{g^{-1}}^{-1}=\{f^{-1}: f\in V_{g^{-1}}\}$. Using the continuity of the inversion operator $f\to f^{-1}$ in $\tau_{\mathrm{co}}$ we have that $V_g$ is open and moreover $g\in V_g\subset\Aut{}(\KB)^c$.

The following proposition says that the inverse of the above observation is true, i.e., for every closed subgroup $G$ of $\Homeo(K)$ there is topological $\LC$-structure $\KB$ on $K$ such that $G=\Aut{}(\KB)$. Moreover, $\KB$ can be taken to be purely dual and projectively ultra-homogeneous.

\begin{proposition}\label{proposition orbits}
Let $G$ be a closed subgroup of $\Homeo(K)$. Then there is a purely dual projective \Fraisse{} structure $\KB$ on domain $K$ such that $\Aut(\KB)=G$. 
\end{proposition}   
\begin{proof}

For every $n>0$, the group $G$ acts on $[n]^K$ in a natural way: for $g\in G$ and $e\in [n]^K$ let 
\[g\cdot e \,\vcentcolon=\, e\circ g^{-1}.\]
We denote this action by $G\acts [n]^K$. Notice that this action corresponds to the following action $G\acts\mathrm{CP}_n$ of $G$ on $\mathrm{CP}_n$: for $g\in G$ and $P=(\Delta_0,\ldots,\Delta_{n-1})\in\mathrm{CP}_n(K)$ let
\[g\cdot P \,\vcentcolon=\, \big(g(\Delta_0),\ldots,g(\Delta_{n-1})\big).\]
For each $n>0$ let $(\mathcal{O}^n_i: i\in I_n)$ be the collection of all orbits of $G\acts [n]^K$.

Consider now the language $\mathcal{L}=\bigcup_{i=1}^\infty\mathcal{L}^n$, where $\mathcal{L}^n$ in the language that consists of $n$-ary relational symbols $\{O^n_i: i\in I_n\}$, one for every orbit  $\mathcal{O}^n_i$.
We turn $K$ into a topological $\LC$-structure $\KB$. For $e\in [m]^K$ we let 
\[\KB\models O^m_i(e)\quad\text{if and only if}\quad e\in\mathcal{O}^m_i.\]
It is immediate that $G\subseteq\Aut(\KB)$.  We work now towards the converse inclusion.

Let $g\in\Aut{}(\KB)$ and let $V(F,U)$ be an open neighborhood of $g$ in $\Homeo{}(K)$. We can assume that $U\neq K$.  We will find $h\in G\cap V(F,U)$ which will prove that $G\supseteq\Aut(\KB)$. Because $g(F)$ is compact and $U$ is a union of clopen sets, $g(F)$ can be covered with finitely many of them, so we can assume without loss of generality that $U$ is clopen and $U\neq K$. Notice that $g\in V\big(g^{-1}(U),U\big)\subset V(F,U)$. Consider the following two dual tuples $e_1,e_2\in 2^K$, with $e_1^{-1}(\{0\})=g^{-1}(U)$, $e_1^{-1}(\{1\})= K\setminus g^{-1}(U)$ and $e_2^{-1}(\{0\})= U$, $e_2^{-1}(\{1\})= K\setminus U$. Since $g$ is an automorphism of $\KB$ and since $e_1=g\cdot e_2$, we have that $e_1$ and $e_2$ lie in the same orbit $\mathcal{O}^2_i$ for some $i\in I_2$. Therefore, there is an $h\in G$ that sends $g^{-1}(U)$ into $U$ and therefore $h\in G\cap V(F,U)$, which proves that $G=\Aut{}(\KB)$.

We prove now that $\KB$ is projectively \Fraisse{}. First notice that for every dual tuple $e\in [m]^K$, there is a unique $i\in I_m$ such that $\KB\models O^m_i(e)$.  Let $\CB\in\KB$ and let $f_1,f_2$ be two epimorphisms of $\KB$ onto $\CB$. We can assume without the loss of generality that $C=\{0,\ldots,m-1\}$ for some $m>0$ and therefore $f_1,f_2\in [m]^K$. Because $f_1$ and $f_2$ induce the same structure $\CB$, there is a unique $i\in I_m$ such that $\KB\models O^m_i(f_1)$ and $\KB\models O^m_i(f_2)$. Therefore, $f_1$ and $f_2$ lie in the same orbit $G\acts [m]^K$, so there is $g\in\Aut{}(\KB)$  such that $f_1\circ g=f_2$, showing that $\KB$ is projectively ultra-homogeneous. 
\end{proof}

\section{Turning a structure to a purely dual one}\label{Section: Turning a structure direct to dual}

Here we show that it is always possible to translate the direct structure into a dual one without losing any information. We provide a counterexample to show that the converse is not always possible. Although purely dual structures are sufficient for the development of the general theory, in Section \ref{Section: Compact spaces} it will be convenient to make use of direct relations. Moreover, there are many examples of structures whose most natural presentation would involve both direct and dual structure.

Let $\LC$ be a language and $\MB$ a topological $\LC$-structure. Let also $s\in\LC$ be a direct relation of arity $n$. For every $k$ with $0<k\leq n$ and for every $f\in [k]^n$ ($f$ is therefore a surjection), we introduce a dual relational symbol $R_s^f$ of arity $k+1$. Let \[\LC_{\not{s}}=\LC\cup\{R_s^f:\, f\in [k]^n \,\,\,\mbox{for some}\,\,\, 0<k\leq n\}\setminus\{s\}.\]
We turn now $\MB$ into an $\LC_{\not{s}}$-structure $\MB_{\not{s}}$ on the same domain $M$. We encode $s^{\MB}$ using the new dual symbols as follows: for $f\in [k]^n$  we let
$\MB_{\not{s}}\models R_s^f(\Delta_0,\ldots,\Delta_{k})$,
if and only if there are $a_0,\ldots,a_{n-1}\in M$ such that    
\[\MB\models s(a_0,\ldots,a_{n-1})\,\,\,\mbox{and}\,\,\, a_i\in\Delta_{f(i)}\,\,\,\mbox{for every}\,\,\,i\in n.\]  

It can easily be checked that $\Aut{}(\MB)$ can be fully recovered from $\Aut{}(\MB_{\not{s}})$, that $\MB$ is projectively \Fraisse{} if and only if $\MB_{\not{s}}$ is, and that  $\Aut{}(\MB)$ and $\Aut{}(\MB_{\not{s}})$ are equal as permutation groups on $M$.  

There are cases of topological $\LC$-structures which can be turned into purely direct structures. However, this is not the case always. The main observation is that if $r$ is direct relation of arity $k$ which belongs to $\LC$  and $\MB$ is a topological $\LC$-structure then $r^{\MB}$ is a set-wise invariant closed subset of $M^k$. Let now $K=2^\N$ and let $\mu$ be the uniform probability measure on $2^\N$. The group $\Aut{}(K,\mu)$ of all continuous measure preserving bijections can be easily seen to be a closed proper subgroup of $\Homeo(K)$ which for every $n>0$ leaves no proper subset of $K^{[n]}$ invariant. 
Therefore, the canonical \Fraisse{} structure given by an application of Proposition \ref{proposition orbits} on $\Aut{}(K,\mu)$ cannot be turned into a purely direct one.

\section{Compact Polish spaces as quotients of dual \Fraisse{} structures}\label{Section: Compact spaces}

We fix a special binary relation symbol $\mathfrak{r}$ whose interpretation will always be a reflexive and symmetric closed relation. A formal relational language $\LC$ will be decorated with the subscript $\mathfrak{r}$ whenever $\mathfrak{r}\in\LC_\mathfrak{r}$. Therefore, an $\LC_\mathfrak{r}$-structure is always going to be a reflexive $\mathfrak{r}$-graph perhaps with some extra structure. We say that an $\LC_\mathfrak{r}$-structure $\KB$ is a pre-space if $\mathfrak{r}^{\KB}$ is moreover transitive and therefore an equivalence relation.

As we noted in the introduction, T. Irwin and S. Solecki used purely direct \Fraisse{} theory to express the pseudo-arc $P$ as a quotient of a projective \Fraisse{} $\{\mathfrak{r}\}$-structure $\mathbb{P}$ via $\mathfrak{r}^{\mathbb{P}}$. Moreover, through their construction, the group $\Aut{}(\mathbb{P})$ naturally embedded in $\Homeo{}(P)$ as a dense subgroup. In \cite{Camerlo}, R. Camerlo characterized all different projective \Fraisse{} classes\footnote{He allows \Fraisse{} classes to lack hereditary property.} of $\{\mathfrak{r}\}$-structures. Their limits are pre-spaces with quotients $M/\mathfrak{r}^{\MB}$ which vary between certain combinations of singletons, Cantor spaces and  pseudo-arcs \cite{Camerlo}. In \cite{Lelek}, D. Barto\v{s}ov\'{a} and A. Kwiatkowska express the Lelek fan $L$ as the quotient of a the projective \Fraisse{} limit $\mathbb{L}$ of a certain class of directed graphs. Their limit $\mathbb{L}$ can be seen again as pre-space in some $\LC_\mathfrak{r}$. Here again the group $\Aut{}(\mathbb{L})$ naturally embedded in $\Homeo{}(L)$ as a dense subgroup.

In this section we show that under the notion of projective \Fraisse{} limit we developed here the same representation applies to every second-countable compact space $Y$.
Since this is trivial for finite spaces, we will restrict ourselves to the case where $Y$ is infinite. 

The following proposition will be used in the proof of Theorem \ref{compactSpacesMainTheorem}.
\begin{proposition}\label{prop_extension}
Let $G,H$ be Polish groups and let $S$ be a dense subgroup of $G$. Then any continuous homomorphism $f:S\to H$ extends to a continuous homomorphism $\tilde{f}:G\to H$. 
\end{proposition}

The proof of Proposition \ref{prop_extension} is an easy exercise given that every Polish admits a compatible left-invariant metric $d$ and given this metric we can define a new compatible complete metric $D$ by $D(x,y)=d(x,y)+d(x^{-1},y^{-1})$. For more details, see page 6 of \cite{Actions}.

\begin{theorem}\label{compactSpacesMainTheorem}
Let $G$ be a closed subgroup of $\Homeo{}(Y)$, for some compact metrizable space $Y$. Then there is a
projective \Fraisse{} pre-space $\KB$ such that $K/\mathfrak{r}^{\KB}$ is homeomorphic to $Y$, and the
quotient projection
\[K \mapsto Y\]
induces a continuous group embedding $\Aut(\KB)\hookrightarrow  G$, with dense image in $G$. 
\end{theorem} 
\begin{proof}

Let $Y$ be an infinite compact Polish space and let $H$ be a countable, dense subgroup of $G$. In what follows, we define a countable Boolean algebra $(\mathcal{F},0_\mathcal{F},1_\mathcal{F},\wedge,\vee,')$ of closed subsets of $Y$ as well as an action of $H$ on $\mathcal{F}$ via Boolean algebra automorphisms. Every set $F\in\mathcal{F}$ will be regular closed. Recall that an open set $U$ is called regular open if $\mathrm{int}(\overline{U})=U$ and a closed set $F$ is called regular closed if $\overline{\mathrm{int}(F)}=F$. We define $0_\mathcal{F},1_\mathcal{F}$ and the operations $\wedge,\vee,'$ as follows:
\begin{itemize}
\item $0_\mathcal{F}=\emptyset$;
\item $1_\mathcal{F}=Y$;
\item $F_1\wedge F_2= \overline{\mathrm{int}(F_1\cap F_2)}$;
\item $F_1\vee F_2= F_1\cup F_2$;
\item $F'=\overline{F^c}$.
\end{itemize}
The boolean algebra axioms are satisfied by the above configuration since $\mathcal{F}$ consists of regular closed sets (see also \cite{Halmos} for the boolean algebra of regular open sets).

To construct the boolean algebra fix first a compatible complete metric $d$ on $Y$. For every $n$ chose a finite open cover $\{V^n_0,\ldots,V^n_{k_n}\}$ of $Y$ such $\mathrm{diam}(V_i^n)<1/n$ for every $i\in\{0,\ldots,k_n\}$. Since $\overline{V}$ is regular closed for every open $V$ we have that the collection $\mathcal{J}=\{F^n_i: F^n_i=\overline{V^n_i},\; n\in\N,\,\,0\leq i\leq k_n \}$ consists of regular closed sets. We define $\mathcal{F}$ to be the least family of closed subsets of $Y$ such that:  
\begin{enumerate}
\item $\mathcal{J}\subset\mathcal{F}$;
\item $\mathcal{F}$ is closed under the boolean operators $\wedge,\vee,'$ and
\item $\mathcal{F}$ is closed under translation by elements of $H$, i.e., if $h\in H$ and $F\in\mathcal{F}$ then $h(F)\in\mathcal{F}$.
\end{enumerate}
Notice that all these operations preserve regularity and since $\mathcal{J}$ and $H$ are countable $\mathcal{F}$ is a countable family of regular closed sets. Notice that this implies that the only $F\in\mathcal{F}$ that has empty interior is the empty set. The group $H$ is acting on $\mathcal{F}$ with Boolean algebra automorphisms: for every $h\in H$ and $F\in\mathcal{F}$ let
\[h\cdot F = h(F).\]

Let $K=\mathrm{S}(\mathcal{F})$ be the Stone space of all ultrafilters $x$ on $\mathcal{F}$. This space comes with a topology whose basic clopen sets can be taken to be the sets of the form $\tilde{F}=\{x : F\in x\}$ for $F\in\mathcal{F}$. The space $K$ is a compact, second-countable, and zero-dimensional. Let $p: K\to Y$ be the natural projection defined by:
\[\big\{p(x)\big\}=\bigcap_{F\in x} F.\]
The map $p$ is continuous surjection with $p(\tilde{F})=F$ for every $F\in\mathcal{F}$. We can turn now $K$ to a $\{\mathfrak{r}\}$-structure $K_{\mathfrak{r}}$ by setting $K_{\mathfrak{r}}\models \mathfrak{r}(x_0,x_1)$ if and only if $p(x_0)=p(x_1)$. It is immediate that $K_{\mathfrak{r}}$ is a pre-space and that $K/\mathfrak{r}^{K_{\mathfrak{r}}}=Y$.

Notice now that $H$ is acting on $K$ with homeomorphisms: for every $h\in H$ and $x\in K$
\[h\cdot x=\big\{h(F): F\in x\big\}\in K.\]
This action is faithful since for every pair $y_0,y_1\in Y$ there are $F_0,F_1\in\mathcal{F}$ such that $y_0\in \mathrm{int}(F_0),y_1\in\mathrm{int}(F_1)$ and $F_0\cap F_1=\emptyset$. Therefore, $H$ embeds into $\Homeo(K)$.
 We will denote this copy of $H$ inside $\Homeo(K)$ by $H_K$  to distinguish it from $H$ which is a subgroup of $\Homeo(Y)$ and we will denote by $T_0$ the inverse of this embedding, i.e.,
\[T_0:H_K\to H\quad\text{with}\quad \widetilde{T_0(h)\big(F\big)}=h(\tilde{F}),\,\,\,\text{for every}\,\,\,F\in\mathcal{F}.\]
The map $T_0$ is also continuous. To see that, let $h\in H_K$ and let $V(L,U)$ be an open neighborhood of $T_0(h)$ in $\Homeo{}(Y)$, i.e., $T_0(h)(L)\subset U$. Since the family  $\{\mathrm{int}(F): F\in \mathcal{J}\}$  constitutes a basis of $Y$ and since $T_0(h)(L)$ is compact, we can find $F_1,\ldots,F_k\in\mathcal{J}$ such that $T_0(h)(L)\subseteq{}F_1\cup\ldots\cup F_k\subseteq{}U$. Let $F_0=F_1\vee\ldots\vee F_k$, then both $F_0$ and $h^{-1}(F_0)$ belong to $\mathcal{F}$. Moreover, $V(\widetilde{h^{-1}(F_0)},\tilde{F_0})$ is an open neighborhood of $h$ in $H_K$ that is mapped via $T_0$ completely inside $V(L,U)$, proving that $T_0$ is continuous at $h$.  

 By applying the Proposition \ref{proposition orbits}, we can endow $K$ with a topological \Fraisse{} structure $\KB_0$ in a purely dual language $\LC$, such that $\overline{H_K}=\Aut{}(\KB_0)$ (the closure here is taken in $\Homeo{}(K)$). By Proposition \ref{prop_extension} the map $T_0$ extends to a continuous homomorphism $T:\Aut{}(\KB_0)\to G$. We denote the image of $\Aut{}(\KB_0)$ under $T$ by $\hat{H}$. Notice that $\hat{H}$ lies densely in $\Homeo{}(K)$ since $H < \hat{H} \le \Homeo{}(K)$, and since the same is true for $H$. Moreover, by the density of $H_K$ in $\overline{H_K}$ the continuity of $T$ and the fact that every $F\in\mathcal{F}$ has non-empty interior we get that for every $h\in\overline{H_K}$ and for every $F\in\mathcal{F}$ the following equality holds
 \begin{equation}
 \widetilde{T(h)\big(F\big)}=h(\tilde{F})\label{1}.
 \end{equation}
 
We combine now the structures $\KB_0$ and $K_{\mathfrak{r}}$ into one $\LC_\mathfrak{r}$-structure $\KB$ on domain $K$, where $\LC_\mathfrak{r}=\mathcal{L}\cup\{\mathfrak{r}\}$. Notice that $\mathfrak{r}^{\KB}$ is invariant under $\Aut(\KB_0)$ since $(x_0,x_1)\in \mathfrak{r}^{\KB}$ if and only if for all $F_0,F_1\in\mathcal{F}$ with $x_0\in \tilde{F}_0$ and $x_1\in \tilde{F}_1$ we have that $F_0\cap F_1\neq \emptyset$.  Thus $\Aut{}(\KB)=\Aut{}(\KB_0)=\overline{H_K}$, every $\AB_0\in\mathrm{Age}(\KB_0)$ uniquely extends to an $\AB\in\mathrm{Age}(\KB)$ and $\KB$ is a also a projective \Fraisse{} structure. The fact that $p(\tilde{F})=F$ for every $F\in\mathcal{F}$ and the relation (\ref{1}) above let us view $T:\Aut{}(\KB)\to G$ as the homomorphism induced by the quotient $p:K\to Y$.

 We are left to show that $T$ is injective. Let $h\in\Aut(\MB)$ so that $h\neq \mathrm{id}_{\Aut{}(\MB)}$. By the continuity of $h$ we can find a non-empty $F$ in $\mathcal{F}$  so that $F\wedge h(F)=\emptyset$. Therefore, the interiors in $Y$ of $p(F)$ and $p\big(h(F)\big)$ do not intersect and because the interior in $Y$ of every non-empty $F$ in $\mathcal{F}$ is non-empty we have that $T(h)\neq \mathrm{id}_{\Homeo(Y)}$.  
\end{proof}

We should remark here that the image $\hat{H}$ of $\Aut{}(\KB)$ under $T$ is in general a meager subset of $G$. This can be seen as follows: first notice that as a corollary of Pettis theorem we have that if $f: B\to D$ is a Baire-measurable homomorphism between Polish groups and $f(B)$ is not meager, then $f$ is open (see for example Theorem 1.2.6 \cite{Actions}). Now notice
that for $F\in\mathcal{F}$ the set $V(\tilde{F},\tilde{F})$ is open in $\Homeo(K)$ but the set  $V(F,F)$ is rarely open in $G$ (except if $Y$ is zero-dimensional or if $G$ contains very few homeomorphisms). Therefore $T$ will fail in general to be an open map.

\end{document}